\documentclass[letterpaper, 10 pt, conference]{ieeeconf}  

\IEEEoverridecommandlockouts                              
\usepackage{graphicx} 
\usepackage{amsmath} 
\usepackage{amssymb}  
\usepackage{amstext,enumerate}
\usepackage{subfigure}
\usepackage{latexsym, color,booktabs}

\newtheorem{assum}{Assumption}

\newtheorem{thm}{Theorem}
\newtheorem{rem}{Remark}
\newtheorem{lem}{Lemma}

\title{\LARGE \bf
Supervisory observer for parameter and state estimation of nonlinear systems using the DIRECT algorithm
}

\author{Michelle S. Chong, Romain Postoyan, Sei Zhen Khong and Dragan Ne\v{s}i\'{c}
\thanks{M. Chong is with the Department of Automatic Control, Lund University, Sweden. Her work is supported by the LCCC Linnaeus Center and the ELLIIT Excellence Center at Lund University.
        {\tt\small michelle.chong@control.lth.se} }%
\thanks{R. Postoyan is with Universit\'e de Lorraine, CRAN, UMR 7039 and the CNRS, CRAN, UMR 7039, France. His work is partially supported by the ANR under the grant
SEPICOT (ANR 12 JS03 004 01).
        {\tt\small romain.postoyan@univ-lorraine.fr} }%
\thanks{S. Khong is with the Institute for Mathematics and its Applications, The University of Minnesota, Minneapolis, MN 55455, USA. {\tt\small szkhong@umn.edu}  }%
\thanks{D. Ne\v{s}i\'{c} is with the Department of Electrical and Electronic Engineering, University of Melbourne, Australia. 
        {\tt\small dnesic@unimelb.edu.au}}%
}

\begin{document}

\maketitle
\thispagestyle{empty}
\pagestyle{empty}

\begin{abstract}
A supervisory observer is a multiple-model architecture, which estimates the parameters and the states of nonlinear systems. It consists of a bank of state observers, where each observer is designed for some nominal parameter values sampled in a known parameter set. A selection criterion is used to select a single observer at each time instant, which provides its state estimate and parameter value. The sampling of the parameter set plays a crucial role in this approach. Existing works require a sufficiently large number of parameter samples, but no explicit lower bound on this number is provided. The aim of this work is to overcome this limitation by sampling the parameter set automatically using an iterative global optimisation method, called DIviding RECTangles (DIRECT). Using this sampling policy, we start with $1+2n_p$ parameter samples where $n_p$ is the dimension of the parameter set. Then, the algorithm iteratively adds samples to improve its estimation accuracy. Convergence guarantees are provided under the same assumptions as in previous works, which include a persistency of excitation condition. The efficacy of the supervisory observer with the DIRECT sampling policy is illustrated on a model of neural populations. 
\end{abstract}

\section{Introduction}
Multiple-model approaches have traditionally been employed in the stochastic setting for estimation algorithms. Strong contenders include the Gaussian sum estimators \cite[Section 8.5]{anderson1979optimal}, the variable structure multiple-model framework \cite{li1996multiple, sheldon1993optimizing} and statistical filters such as particle filters \cite{arulampalam2002tutorial} and unscented Kalman filters \cite{wan2000unscented}. In the deterministic domain, the focus has been mainly on improving robustness of adaptive schemes for control purposes \cite{narendra2011changing, anderson2000multiple}, \cite[Chapter 6]{liberzon2003switching}. In the context of \emph{parameter} and \emph{state} estimation for dynamical systems, the multiple-model approach provides an alternative to other online estimation algorithms, such as adaptive observers \cite[Chapter 7]{besancon2007nonlinear}. The idea is to design a bank of \emph{state} observers (also known as a multi-observer), in which the state observer is designed for some sampled parameter values in a given set. Parameter and state estimates are then derived by combining the information from a subset of the state observers. This approach has the added benefit of modularity in the separation of the state and parameter estimation problems. This has been pursued for continuous-time linear systems by  \cite[Section 8.5]{anderson1979optimal, li1996multiple} in the stochastic setting, and by \cite{aguiar2008identification, narendra2011changing} in the deterministic setting, to name a few references. Most recently, we extended the deterministic setup to continuous-time nonlinear systems in \cite{chong2015parameter} inspired by results in supervisory control \cite[Chapter 6]{liberzon2003switching}, which we called a \emph{supervisory observer}.

In \cite{chong2015parameter}, the unknown system parameters are assumed to be constant and to belong to a known compact parameter set. Finite samples are drawn from this parameter set and a state observer is designed for each sample in a way that is robust against parameter mismatches. A selection criterion based on the mismatch between the estimated output of each observer and measured output from the plant then provides the state and parameter estimates at any given time. We showed in \cite{chong2015parameter} that the estimates are guaranteed to converge to be within a required margin, as long as the parameter set is sufficiently densely sampled and a persistency of excitation condition holds. To potentially ease the need for a large number of samples, we also introduced a dynamic sampling scheme which updates the sampling of the parameter set by iteratively zooming in on the region of the parameter set where the true parameter is `most likely' to reside. However, a major drawback is that the user is required to choose a pre-determined number of parameter samples at the start of the algorithm, which is hard to estimate. 

In this paper, we aim to overcome this major drawback by automatically sampling the parameter set using an iterative global optimisation method for Lipschitz cost functions on a compact domain, called DIRECT (DIviding RECTangles) which was initially proposed in \cite{jones1993lipschitzian}. This sampling procedure starts from sampling the center point of the parameter set and  additional given sample points in its neighbourhood. At subsequent iterations, DIRECT takes additional samples.  This is achieved by dividing the domain into non-overlapping hyperrectangles and sampling the center of each one. A lower bound on the cost function in each hyperrectangle is obtained using the fact that the cost function is Lipschitz.  The search at the local level is achieved by identifying potentially optimal hyperrectangles based on its lower bound of the cost function and its size, which are then further divided. In other words, the algorithm automatically generates additional samples in potentially optimal regions. The procedure is stopped once a pre-calculated number of iterations is reached. 

In our supervisory observer, the cost function used by the DIRECT algorithm is an integral form of the output mismatch of each observer and the plant over a finite time interval, which we call the monitoring signal. The domain of the cost function is the compact parameter set to which the true parameter belongs. We wait a sufficiently long interval between each iteration to allow the transient effects of each state observer to decay in order to obtain an improved accuracy. More observers are added as DIRECT updates the sampling of the parameter set. Upon the termination of DIRECT, we implement the last chosen observer for the remainder of the algorithm's run-time. Contrary to the old dynamic sampling policy in \cite{chong2015parameter}, the DIRECT sampling policy provides the following improvements. First, DIRECT does not require the user to estimate the number of observers needed beforehand. Instead, we start with $1+2n_p$ observers, where $n_p$ is the dimension of the parameter set. Then, DIRECT automatically takes samples until a given termination time, which is a clear advantage over the old sampling policy from the user's perspective. Second, the DIRECT policy potentially eases computational burden as we no longer need to implement all the required number of observers in parallel from initialisation. Third, after a pre-computed time, the supervisory observer implements only one observer for the rest of the run-time and thereby further reducing the computational resources required. 

The DIRECT algorithm was employed in the context of extremum-seeking in \cite{khong2013multidimensional} where the global minimum of the steady-state input-output map of a nonlinear time-invariant dynamical system is found using DIRECT without knowledge of the system model. Our setup is `gray box' in nature, where the structure of the parameterised nonlinear dynamical system is known, but not the states and parameters. The supervisory observer we propose can be viewed as the problem of online extremisation in dynamical systems, where the estimation algorithm aims to provide estimates that minimise the estimation error based on the steady-state behaviour of each observer by waiting sufficiently long between sampling times.  

The paper is structured as follows. We introduce the notation in Section \ref{sec:prelim} and formulate the problem in Section \ref{sec:prob}. Section \ref{sec:sup_obs} describes the supervisory observer with the DIRECT sampling policy in detail and we provide convergence guarantees in Section \ref{sec:conv}. We illustrate the efficacy of the proposed algorithm in Section \ref{sec:example} by revisiting a model of neuron populations considered in \cite{chong2015parameter}. Section \ref{sec:conc} concludes the paper with some discussions for future work. 

\section{Preliminaries} \label{sec:prelim}
Let $\mathbb{R}=(-\infty,\infty)$, $\mathbb{R}_{\geq 0}=[0,\infty)$, $\mathbb{R}_{>0}=(0,\infty)$, $\mathbb{N}=\{0,1,\dots\}$ and $\mathbb{N}_{\geq 1}=\{1,2,\dots\}$. Let $(u,v)$ where $u\in\mathbb{R}^{n_u}$ and $v\in\mathbb{R}^{n_v}$ denote the vector $(u^T,v^T)^{T}$. 
The smallest integer greater than $v\in\mathbb{R}$ is denoted by $\lceil v \rceil$.
For a vector $x=(x_1,x_2,\dots,x_n)\in\mathbb{R}^{n}$, the $\infty$-norm of $x$ is denoted $|x|:={\max}\{|x_1|,|x_2|,\dots,|x_n|\}$.
Let $\mathcal{H}(p_c)$ denote the hypercube centered at $p_c\in\mathbb{R}^{n_p}$ where the distance of $p_{c}$ to the edge is $1$, i.e. $\mathcal{H}(p_c):=\{p\in\mathbb{R}^{n_p}: |p-p_{c}| \leq 1 \}$. Hence, $\Delta\mathcal{H}(0)$ is the hypercube with center point at the origin and its distance to the edge is $\Delta$, i.e. ${\Delta}\mathcal{H}(0):=\{p\in\mathbb{R}^{n_p}: |p| \leq \Delta \}$. 
For any $\Delta>0$, the set of piecewise continuous functions from $\mathbb{R}_{\geq 0}$ to $\Delta\mathcal{H}(0)$ is denoted $\mathcal{M}_{\Delta}$.
A continuous function $\alpha:\mathbb{R}_{\geq 0}\to\mathbb{R}_{\geq 0}$ is a class $\mathcal{K}$ function, if it is strictly increasing and $\alpha(0)=0$; additionally, if $\alpha(r)\to\infty$ as $r\to\infty$, then $\alpha$ is a class $\mathcal{K}_{\infty}$ function. A continuous function $\beta:\mathbb{R}_{\geq0}\times \mathbb{R}_{\geq 0} \to \mathbb{R}_{\geq 0}$ is a class $\mathcal{KL}$ function, if: (i) $\beta(.,s)$ is a class $\mathcal{K}$ function for each $s\geq 0$; (ii) $\beta(r,.)$ is non-increasing and (iii) $\beta(r,s)\to 0$ as $s\to 0$ for each $r\geq 0$.

\section{Problem formulation} \label{sec:prob}

Consider the following nonlinear system
\begin{equation}
	\begin{array}{lll}
			\dot{x} & = & f(x,p^\star,u), \\
			y & = & h(x,p^\star),			\label{eq:system}
	\end{array}	
\end{equation}
where the state is $x\in\mathbb{R}^{n_x}$, the measured output is $y\in\mathbb{R}^{n_y}$, the measured input is $u\in\mathbb{R}^{n_u}$ and the \emph{unknown} parameter vector $p^{\star}\in\Theta\subset\mathbb{R}^{n_p}$ is constant. We assume that $\Theta$ is a known, normalised unit hypercube\footnote{Any compact parameter set $\Theta$ can be embedded in a hyperrectangle and then normalised to be a unit hypercube. Hence, the assumption of $\Theta$ being a hypercube of edge $1$ is made with no loss in generality as long as the parameter $p^\star$ lies in a known (potentially large) compact set.} $\mathcal{H}(p_c)$. The function $f:\mathbb{R}^{n_x}\times \mathbb{R}^{n_u} \times \mathbb{R}^{n_p} \to \mathbb{R}^{n_x}$ is locally Lipschitz and $h:\mathbb{R}^{n_x}\times \mathbb{R}^{n_p} \to \mathbb{R}^{n_y}$ is continuously differentiable. 

We aim to estimate the parameter $p^{\star}$ and the state $x$ of system \eqref{eq:system} assuming that the output $y$ and the input $u$ are measured. To this end, we use the supervisory observer architecture we proposed in \cite{chong2015parameter}. The idea is the following: the parameter set $\Theta$ is sampled and a state observer is designed for each parameter sample, forming a bank of observers. One observer is chosen to provide its state estimate and parameter value at any time instant based on a criterion. The results in \cite{chong2015parameter} showed that both estimates converge to within a tunable margin of the corresponding true values provided that the number of samples is sufficiently large, under some conditions. The sampling policies in \cite{chong2015parameter} require a sufficiently large number of samples and no explicit lower bound on this number was provided in \cite{chong2015parameter}. The objective of this paper is to overcome these issues by sampling the parameter set $\Theta$ dynamically in a smarter manner using the DIRECT optimisation algorithm, such that the number of samples is automatically generated by the algorithm based on the desired estimation accuracy. In other words, the user no longer has to set the number of samples needed from the start, as the algorithm automatically achieves the required number of samples after some iterations and terminates at a pre-computed time. 

We maintain all assumptions made in the original setup in \cite{chong2015parameter}. In particular, we assume the following boundedness property for system \eqref{eq:system}.

\begin{assum} \label{assum:bounded_x}
For any initial condition $x(0)$ and any piecewise-continuous input $u$, system \eqref{eq:system} is forward complete and generates unique uniformly bounded solutions, i.e. for any $\Delta_x$, $\Delta_u \geq 0$, there exists a constant $K_x = K_x(\Delta_x,\Delta_u) > 0$ such that for all $x(0)\in \Delta_{x}\mathcal{H}(0)$ and $u\in \mathcal{M}_{\Delta_u}$, the corresponding unique solution to \eqref{eq:system} satisfies $|x(t)| \leq K_{x} $, for all $t\geq 0$. \hfill $\Box$
\end{assum}

\section{Supervisory observer with the DIRECT sampling policy} \label{sec:sup_obs}
In this section, we first recall the supervisory observer architecture proposed in \cite{chong2015parameter}. We then explain how DIRECT is implemented to sample the parameter set $\Theta$.

\subsection{Architecture} \label{sec:arch}
The sampling of the parameter set $\Theta$ is carried out iteratively at each update time instant $t_k$, $k\in\mathbb{N}$ satisfying
\begin{equation} 
	t_{k+1}- t_{k} =:T_d, \label{eq:def_Td}
\end{equation} 
where $T_d>0$ is a design parameter. 

Let $\hat{P}(k)$ denote the set of parameter sample points generated at time $t_k$, $k\in\mathbb{N}_{\geq 1}$; and the generation of these points will be described in Section \ref{sec:sup_direct}. Also, let $\hat{\Theta}(k):=\hat{\Theta}(k-1)\cup \hat{P}(k)$ be the set of all parameter sample points from $t_0$ to $t_k$, where their corresponding cardinalities are ${N}_{\hat{P}}(k)$ and $N_{\hat{\Theta}}(k)$, respectively. Consequently, we have that $\hat{\Theta}(k-1)$ is a subset of $\hat{\Theta}(k)$. At the initial time $t_0$, $\hat{P}(0)$ and $\hat{\Theta}(0)$ consist of the centre point of $\Theta$ and $2n_p$ additional points near it, i.e. $\hat{P}(0)=\hat{\Theta}(0)=\{p_c, p_{c}\pm\delta e_i\}$, where $\delta= 1 / 3$ and $e_i$ is the $i$-th unit vector of $\mathbb{R}^{n_p}$. Hence, $N_{\hat{\Theta}}(0)={N}_{\hat{p}}(0)=1+2n_p$. In other words, the number of observers implemented over the first interval of time $[t_0,t_1)$ is $1+2n_p$. At subsequent update time instants $t_k$, $k\in\mathbb{N}$, new sampling points are added and an observer is designed for each of the newly generated samples $p_{i}^{k} \in \hat{P}(k)$, for $i\in\{1,\dots,{N}_{\hat{P}}(k)\}$, as follows
\begin{align}
	\dot{\hat{x}}_{i} &= \hat{f}(\hat{x}_{i}, p_{i}^k,u,y), \qquad \forall t\in[t_{k},t_{k+1}),  \nonumber\\
	\hat{y}_{i} &= h(\hat{x}_{i}, p_i^k),  \label{eq:multiobs}
\end{align}
where the function $\hat{f}$ is continuously differentiable. At initialisation $t_0=0$, we set arbitrary initial conditions $\hat{x}_{i}(t_0) \in \mathbb{R}^{n_x}$ for $i\in \{1,\dots,1+2n_p\}$. At subsequent update times $t_k$, $k\in\mathbb{N}_{\geq 1}$, each `old' observer is kept running, i.e. the observers designed for each $p\in\hat{\Theta}(k-1)$, and each new observer is initialised as follows
\begin{equation}
		\hat{x}_{i}(t_k^{+}) = \hat{x}_{\sigma(t_k)}(t_{k}), \label{eq:multiobs2}
\end{equation} 
where $\sigma$ chooses one observer from the bank of observers and is defined below in \eqref{eq:decision_monitor}.

\begin{rem}
		We only initialise the `new' observers according to \eqref{eq:multiobs2}. Our results also apply when all the observers are reinitialised according to \eqref{eq:multiobs2}. However, no significant advantage is seen both in the analysis and when implemented in simulations in Section \ref{sec:example}. \hfill $\Box$
\end{rem}

Denoting the state estimation error as $\tilde{x}_{i}:=\hat{x}_{i}-x$, the output error as $\tilde{y}_{i}:=\hat{y}_{i}-y$ and the parameter error as $\tilde{p}_{i}^{k}:=p_{i}^{k}-p^{\star}$, for all $i\in\{1,\dots,N_{\hat{\Theta}}(k)\}$, we obtain the following state estimation error systems for all $t\in[t_{k},t_{k+1})$,
{ \begin{align}
	\dot{\tilde{x}}_{i} & =  \hat{f}(\tilde{x}_{i}+x,\tilde{p}_{i}^{k}+p^{\star},u,y) - f(x,p^{\star},u) \nonumber \\
	& =: F(\tilde{x}_{i}, \tilde{p}_{i}^{k}, p^{\star}, u, x),  \nonumber \\
	\tilde{y}_{i} & =  h(\tilde{x}_{i}+x,\tilde{p}_{i}^{k}+p^{\star}) - h(x,p^{\star}) =: H(\tilde{x}_{i},  \tilde{p}_{i}^{k}, p^{\star},x), \nonumber \\
	\tilde{x}_{i}(t_k^+) & = \left\{\begin{array}{ll} \hat{x}_{\sigma(t_k)}(t_k) - x(t_k), & i\in\hat{P}(k), \\ 
												  \hat{x}_{i}(t_k) - x(t_k), & i\in\hat{\Theta}(k)\setminus \hat{P}(k). \end{array}  \right. \label{eq:error_sys_general}
\end{align} } 
All the observers are designed such that the following property holds.
\begin{assum}\label{ass:obs_error_i}
	Consider the state estimation error system \eqref{eq:error_sys_general} for $i\in\{1,\dots,N\}$, $N\in\mathbb{N}_{\geq 1}$. Let $\tilde{\Theta}:=\{ p-p^{\star} : p \in \Theta \textrm{ and } p^{\star} \in \Theta  \}$. There exist scalars $a_1$, $a_2$, $\lambda_0>0$ and a continuous non-negative function $\tilde{\gamma}:\tilde{\Theta}\times \mathbb{R}^{n_x} \times \mathbb{R}^{n_u} \rightarrow \mathbb{R}_{\geq 0}$ where $\tilde{\gamma}(0,z,\bar{z})=0$ for all $z \in \mathbb{R}^{n_x}$, $\bar{z} \in \mathbb{R}^{n_u}$, such that for any $\tilde{p}_{i} \in \tilde{\Theta}$, there exists a continuously differentiable function $V_{i}: \mathbb{R}^{n_x} \rightarrow \mathbb{R}_{\geq 0}$ which satisfies the following for all $u\in\mathbb{R}^{n_u}$, $\tilde{x}_{i} \in\mathbb{R}^{n_x}$, $x\in\mathbb{R}^{n_x}$
	\begin{equation}
		a_{1}|\tilde{x}_{i}|^{2} \leq V_{i}(\tilde{x}_{i}) \leq a_{2}|\tilde{x}_{i}|^{2}, \label{eq:lyap_iss_1}
	\end{equation}
	\begin{equation}
			\frac{\partial V_i}{\partial \tilde{x}_{i}} F(\tilde{x}_{i},\tilde{p}_{i},p^{\star},u,x) \leq -\lambda_{0} V_{i}(\tilde{x}_{i}) + \tilde{\gamma}(\tilde{p}_{i},x,u). \label{eq:lyap_iss_2}
	\end{equation}	 \hfill $\Box$
\end{assum}
When there is no parameter mismatch $\tilde{\gamma}(0,x,u)=0$, Assumption \ref{ass:obs_error_i} implies that state estimates converge exponentially to the true state for all initial conditions. When there is a parameter mismatch $\tilde{p}_i\neq 0$, the state estimation error system satisfies an input-to-state exponential stability property with respect to $\tilde{p}_i$ in view of Assumption \ref{assum:bounded_x}. Examples of system \eqref{eq:system} for which observers \eqref{eq:multiobs} can be designed are provided in Section VI of \cite{chong2015parameter}. This includes linear systems and a class of nonlinear systems with monotone nonlinearities, such as the neural example considered later in simulations in Section \ref{sec:example}.

We assume that the output error of each of the observer $\tilde{y}_{i}$ satisfies the following property.
\begin{assum} \label{ass:PE_y_tilde}
	Consider the state estimation error system \eqref{eq:error_sys_general} for $i \in \{1,\dots,N\}$, $N\in\mathbb{N}_{\geq 1}$. For all $\Delta_{\tilde{x}}$, $\Delta_{x}$, $\Delta_{u}>0$, there exist a constant $T_{f}=T_{f}(\Delta_{\tilde{x}}, \Delta_{x}, \Delta_{u})>0$ and a class $\mathcal{K}_{\infty}$ function $\alpha_{\tilde{y}}=\alpha_{\tilde{y}}(\Delta_{\tilde{x}}, \Delta_{x}, \Delta_{u})$ such that for all $\tilde{x}_{i}(0)\in \Delta_{\tilde{x}}\mathcal{H}(0)$, $x(0) \in \Delta_{{x}}\mathcal{H}(0)$, for any $u\in \mathcal{M}_{\Delta_u}$,  and for all $\tilde{p}_{i} \in \tilde{\Theta}$, the corresponding solution to \eqref{eq:error_sys_general} satisfies
	\begin{equation}
		\int_{t-T_{f}}^{t} |\tilde{y}_{i}(\tau)|^{2} d\tau \geq \alpha_{\tilde{y}}(|\tilde{p}_{i}|), \qquad \forall t\geq T_f. \label{eq:tilde_y_PE}
	\end{equation} \hfill $\Box$
\end{assum}
The inequality \eqref{eq:tilde_y_PE} is a variant of the persistency of excitation (PE) condition found in many adaptive and identification schemes \cite{bitmead1984persistence}. In \cite[Proposition 1]{chong2015parameter}, we relate the PE-like condition \eqref{eq:error_sys_general} to the classical PE condition in the literature such that Assumption \ref{ass:PE_y_tilde} can be guaranteed a priori for certain classes of systems.  

The output error from each observer forms the monitoring signal used in the DIRECT sampling algorithm, defined as follows
	 \begin{equation}
			 \mu(p,t_k,t) := \int_{t_k}^{t} \exp(-\lambda(t-s)) |\tilde{y}(p, s)|^{2} ds, \qquad  \forall t\geq t_k, \label{eq:monitoring}
		 \end{equation}
		where $\lambda>0$ is a design parameter and we use $\tilde{y}(p,t)$ in place of $\tilde{y}_{i}(t)$ to highlight its dependance on the parameter $p$. Our parameter and state estimates are chosen from the bank of observers to be, for any $t\geq 0$,
\begin{align}
		\hat{p}(t) &:= \hat{p}_{\sigma(t)}, \\
		\hat{x}(t) &:= \hat{x}_{\sigma(t)}(t), \label{eq:estimates}
\end{align}
where $\sigma$ is given by
\begin{align}
		\sigma(t) & \in \underset{i\in\{1,\dots,N_{\hat{\Theta}}(k)\}}{\arg \min} \mu(p_i,t_k,t),\;  \forall t\in[t_{k}, t_{k+1}). \label{eq:decision_monitor}
\end{align}

We will see in Section \ref{sec:terminate} that the DIRECT sampling policy terminates according to a criterion described in Section \ref{sec:terminate} at $t_{k^{\star}}$, $k^{\star}\in\mathbb{N}$, after which we only implement the observer for the last chosen sample at $p_{\sigma(t_{k^{\star}})}$, i.e the supervisory observer is reduced to one observer as follows, for all $t\in[t_{k^{\star}},\infty)$,
\begin{align}
	\dot{\hat{x}}_{\sigma(t_{k^{\star}})} &= \hat{f}\left(\hat{x}_{\sigma(t_{k^{\star}})}, p_{\sigma(t_{k^{\star}})},u,y\right)  \nonumber\\
	\hat{y}_{\sigma(t_{k^{\star}})} &= h\left(\hat{x}_{\sigma(t_{k^{\star}})}, p_{\sigma(t_{k^{\star}})} \right). \label{eq:multiobs_last}
\end{align}

\subsection{The DIRECT sampling policy} \label{sec:sup_direct}
We generate the sampling points of the parameter set $\Theta$ at every update time instant $t_k$, $k\in\mathbb{N}$ according to the DIRECT sampling policy as follows 

\begin{enumerate}[i.]
\item At the initial time $t_0=0$, set $k=0$ (iteration counter).
\item  At $t=t_1$, evaluate $\mu(p_c,t_0,t_1)$, where we recall that $p_c$ is  the center point of $\Theta$ and follow the procedure for dividing the hyperrectangle\footnote{Procedure for dividing a hyperrectangle: Given a hyperrectangle at time $t_k$ for $k\in\mathbb{N}$, identify the dimensions $i\in I \subseteq \{1,\dots,n_p\}$ in which the hyperrectangle has the maximum edge length and let $\delta$ be a third of this value.  Divide the hyperrectangle containing the sample point $p_j$ into thirds according to the dimensions in $j\in I$, in ascending order of $\min\{\mu(p_j - \delta e_i,t_{k-1},t_{k}), \mu(p_j + \delta e_i,t_{k-1},t_k) \}$, where $e_i$ is the $i$-th unit vector.}.  Set $\hat{\mu}_{1}=\mu(p_c,t_0,t_1)$. Increment the iteration counter to $k = 1$.
\item Identify the set $S(k)$ of the indices of potentially optimal hyperrectangles, i.e. for every $i\in\{1,\dots,N(k)\}$ samples for which there exists ${L}_{k}>0$ such that,
		\begin{align}
				& \mu(p_i,t_{k-1},t_k) - {L}_{k} d_i \leq  \mu(p_j,t_{k-1},t_k) - {L}_{k} d_j, \nonumber \\ & \qquad \qquad \qquad \qquad \qquad \qquad j\in\{1,\dots,N(k)\},  \nonumber \\
				& \mu(p_i,t_{k-1},t_k) - {L}_{k} d_i \leq \hat{\mu}_{k-1} - \epsilon |\hat{\mu}_{k-1}|, \;  \epsilon >0, \label{eq:PO_cond2}
		\end{align}
		where $d_i$ denotes the distance from the centre point $p_i$ to a vertex of the hyperrectangle $i$. Remarks \ref{rem:PO_rect_L} and \ref{rem:PO_rect_e} below discuss the search parameters $L_k$ and $\epsilon$ in \eqref{eq:PO_cond2}.  
\item For each potentially optimal hyperrectangle indexed by $j\in S(k)$, subdivide the hyperrectangle indexed by $j$ according to the procedure for dividing hyperrectangles\footnotemark[2].
\item Increment the iteration counter, $k^+=k+1$ and set the estimate
		\begin{equation}
				\hat{\mu}_{k} := \underset{i\in\{1,\dots,N(k)\}}{\min} \mu(p_i,t_{k-1},t_k). \label{eq:muhat}
		\end{equation}
\item Go to Step iii until $k^*$ iterations have been reached. The selection of $k^*$ is specified in Section \ref{sec:terminate}.  
\end{enumerate}

\begin{rem} \label{rem:PO_rect_L}
The search parameter $L_{k}$ in \eqref{eq:PO_cond2} can be thought of as a rate-of-change constant. In the case where the Lipschitz constant $\tilde{L}_{k}$ of function $\mu$ is known, the user may restrict the search for $L_{k}$ to $L_{k}\leq \tilde{L}_{k}$. However, the knowledge of Lipschitz constant is not a requirement for convergence to the global minimum and efficient algorithms can be used to find $L_{k}$ such as one called Graham's scan, see \cite{jones1993lipschitzian}.	\hfill $\Box$
\end{rem}

\begin{rem} \label{rem:PO_rect_e}
The search parameter $\epsilon>0$ in \eqref{eq:PO_cond2} ensures that at the current iteration $k$, only hyperrectangles with cost $\mu$ that is much smaller than the minimum cost of the previous iteration $\hat{\mu}_{k-1}$ are identified as potentially optimal. Computational results in \cite{jones1993lipschitzian} show that DIRECT is insensitive to the choice of $\epsilon$ and a good value for $\epsilon$ ranges from $10^{-3}$ to $10^{-7}$.	\hfill $\Box$
\end{rem}

In practice, the DIRECT algorithm can be implemented with modifications of the code for DIRECT from \cite{finkelcode}. This is due to the dynamic cost $\mu$ (c.f. \eqref{eq:monitoring}) in our setup as opposed to the static cost the code was originally written for.

\subsection{Termination criterion of the DIRECT sampling policy} \label{sec:terminate}
The final piece of the algorithm is the termination time of the DIRECT sampling policy. Before doing so, a critical sampling property required to show convergence is the fact that DIRECT will generate samples such that the distance between the true parameter $p^{\star}$ and the closest sample in $\hat{\Theta}(k)$ as defined below
\begin{equation}
		d(p^{\star}, \hat{\Theta}(k)) := \underset{p\in\hat{\Theta}(k)}{\min} |p^{\star} - p| \label{eq:d_def}
\end{equation}
tends to zero if $N_{\hat{\Theta}(k)}$ tends to infinity with increasing $k$. We formalize this in the next lemma.

\begin{lem} \label{lem:d_0}
	The DIRECT method of sampling the parameter set $\Theta$ results in the following property,
	\begin{align}
		N(k) \to \infty \implies	d(p^{\star},\hat{\Theta}(k)) \to 0, \label{eq:lem_d}
	\end{align}
	where we recall that $\hat{\Theta}(k)$ is the set of all sample points at time $t\in[t_{k},t_{k+1})$. \hfill $\Box$
\end{lem}
\begin{proof}
	The proof follows from the arguments provided in Section 5 of \cite{jones1993lipschitzian}, which we recapitulate here. Suppose to the contrary that $d(p^{\star},\hat{\Theta}(k))\not\to 0$ as $k\to\infty$, then there must exist $d^{*}>0$ such that $\underset{k\to\infty}{\lim}d(p^{\star},\hat{\Theta}(k)) = d^*$ since $d(p^{\star},\hat{\Theta}(k))$ is non-increasing with $k$ and lower-bounded by $0$. In other words, letting $r_k$ be the smallest number of divisions undergone by any hyperrectangle at iteration $k$, this means that there exists $r_{k^*} \in \mathbb{N}$ in which the number of divisions never increases after iteration $k^*$, i.e. $\underset{k\to\infty}{\lim} r_{k} = r_{k^*}$. Therefore, at iteration $k^*$, there will be at least one hyperrectangle with $r_{k^*}$ divisions forming the set $S_{r^*}$. Let the cardinality of this set be $N_{r^*}$. All hyperrectangles in $S_{r^*}$ have the largest center-to-vertex distance $d_{r^*}$, but may differ in the value of the monitoring signal $\mu$. According to the conditions for potentially optimal hyperrectangles in \eqref{eq:PO_cond2}, the hyperrectangle $j\in S_{r^*}$ with the best value will be identified as a potentially optimal hyperrectangle. Since hyperrectagle $j$ is potentially optimal, it will be divided. By iteration $k^*+N_{r^*}$, all the hyperrectangles in set $S_{r^*}$ would have been divided. This contradicts the assumption that $\underset{k\to\infty}{\lim} r_{k} = r_{k^*}$. Therefore, this proves that $\underset{k\to\infty}{\lim} r_{k} = \infty$ and consequently, we obtain \eqref{eq:lem_d}.   
\end{proof} 

Lemma \ref{lem:d_0} can be used to derive a termination time for the DIRECT sampling policy. This is the purpose of the next lemma which follows from Theorem 4.2 in \cite{gablonsky2001thesis} and Lemma \ref{lem:d_0}.

\begin{lem} \label{lem:direct_rect}
	Given any $d^{\star}>0$, let $k^* := 3^{n_p -1}\left(\frac{3^{n_p(i+1)}-1}{3^{n_p}-1}\right)$, where $i$ satisfies
	\begin{equation}
		\frac{(n_p 3^{-2i})^{1/2}}{2} \leq d^{\star}.
	\end{equation}
	Then, the DIRECT algorithm described in Section \ref{sec:sup_direct} samples the parameter set $\Theta \subset \mathbb{R}^{n_p}$ such that
	\begin{equation}
		d(p^{\star},\hat{\Theta}(k)) \leq d^{\star}, \; \forall k\geq k^*. \label{eq:p_resolution}
	\end{equation}
	  \hfill $\Box$
\end{lem}

Hence, given a desired bound $d^{\star}>0$ on the distance between the true parameter $p^*$ and the closest sampling point, the algorithm can be terminated once the number of iterations reaches $k^{*}$ as defined in Lemma \ref{lem:direct_rect}. In practice, once we have decided on $d^\star$, we know that after $k^\star T_d$ units of time, a single observer can be run as described at the end of Section \ref{sec:arch}. 

\begin{rem} \label{rem:k_iteration}
	 The estimation of the number of iterations $k^{*}$ to achieve the desired resolution \eqref{eq:p_resolution} in Lemma \ref{lem:direct_rect} is calculated based on the assumption that only one hyperrectangle gets divided at each iteration. In reality, the number of hyperractangles identified to be potentially optimal according to \eqref{eq:PO_cond2} can be more than one (see Section \ref{sec:example}) and hence, the calculations in Lemma \ref{lem:direct_rect} is an over-approximation. The calculation is tight only when the cost function $\mu(p,t,\tau)$ is constant for all $p\in\Theta$ and $t$, $\tau>0$. \hfill $\Box$
\end{rem}

\section{Convergence guarantees} \label{sec:conv}
\subsection{Main result} \label{sec:main_result}
We provide the following convergence guarantees and its proof is provided in Section \ref{sec:proof_main}.
\begin{thm} \label{thm:direct}
	Consider system \eqref{eq:system}, the multi-observer \eqref{eq:multiobs}-\eqref{eq:multiobs2} and \eqref{eq:multiobs_last}, the monitoring signals \eqref{eq:monitoring}, the selection criterion \eqref{eq:decision_monitor} and the estimates \eqref{eq:estimates}, under Assumptions \ref{assum:bounded_x}-\ref{ass:PE_y_tilde} and the DIRECT sampling policy. Given any $\Delta_{x}$, $\Delta_{\tilde{x}}$, $\Delta_u>0$, $d^{\star}>0$ and  $\eta>0$, there exist a class $\mathcal{K}_{\infty}$ function ${\gamma}_{\tilde{x}}$, a sufficiently large $T>0$ such that for any sampling interval $T_d \geq T$, a class $\mathcal{K}_{\infty}$ function $\nu_{\tilde{p}}$ and a constant $T^{\star}>0$ such that the following holds
	\begin{align}
		&|\tilde{p}_{\sigma(t)}(t)| \leq \nu_{\tilde{p}}(d^{\star}) + \eta, \qquad \forall t\geq T^{\star}, \nonumber \\ &\underset{t\to\infty}{\limsup} |\tilde{x}_{\sigma(t)}(t)| \leq {\gamma}_{\tilde{x}}(d^{\star}) + \eta, \label{eq:conc}
	\end{align}
	for all $(x(0),\tilde{x}(0))\in \Delta_{x}\mathcal{H}(0) \times \Delta_{\tilde{x}}\mathcal{H}(0)$ and for any $u\in\mathcal{M}_{\Delta_u}$.  \hfill $\Box$
\end{thm}

The convergence guarantees \eqref{eq:conc} show that the upper bound of the estimation error on the parameter and states decreases with $d^\star$, after a sufficiently long time. Additionally, note that $\eta$ can be taken to be as small as desired by increasing $T$. Hence, the estimation accuracy can be tuned by adjusting $d^\star$ and by ensuring a sufficiently large sampling time $T_d$.  

\subsection{Proof of Theorem \ref{thm:direct}} \label{sec:proof_main}
We first prove several lemmas which are used in showing convergence of the algorithm before proving Theorem \ref{thm:direct}.

\begin{lem}[Lemma 1 in \cite{chong2015parameter}] \label{lem:state_error}
	Consider system \eqref{eq:system} and the state error system \eqref{eq:error_sys_general} under Assumption \ref{ass:obs_error_i}. There exist constants $\bar{k}$, $\bar{\lambda}>0$ such that for any $\Delta_{x}$, $\Delta_{\tilde{x}}$, $\Delta_u>0$, there exists a class $\mathcal{K}_{\infty}$ function $\bar{\gamma}_{\tilde{x}}$ such that for any $p$, $p^{\star}\in\Theta$, $(x(0),\tilde{x}(0))\in \Delta_{x}\mathcal{H}(0) \times \Delta_{\tilde{x}}\mathcal{H}(0)$ and $u\in\mathcal{M}_{\Delta_u}$, the solution to \eqref{eq:error_sys_general} satisfies
	\begin{equation}
			|\tilde{x}(t)| \leq \bar{k} \exp(-\bar{\lambda}t) |\tilde{x}(0)| + \bar{\gamma}_{\tilde{x}}(|\tilde{p}|), \qquad \forall t\geq 0,
	\end{equation}
	where $\tilde{p}:=p-p^{\star}$. \hfill $\Box$
\end{lem}

\begin{lem} \label{lem:mu}
	Consider system \eqref{eq:system}, the state error system \eqref{eq:error_sys_general}, for $i\in\{1,\dots,N\}$, $\mathbb{N}_{\geq 1}$ and the monitoring signal \eqref{eq:monitoring} under Assumption \ref{assum:bounded_x}, \ref{ass:obs_error_i} and \ref{ass:PE_y_tilde}. For any $\Delta_{x}$, $\Delta_{\tilde{x}}$, $\Delta_u>0$ and $\epsilon_{\mu}>0$, there exist class $\mathcal{K}_{\infty}$ functions $\bar{\alpha}$ and $\underline{\alpha}$ independent of $\epsilon_{\mu}$ and a $T=T(\Delta_{x}, \Delta_{\tilde{x}}, \Delta_u, \epsilon_{\mu})>0$ such that for all $p$, $p^{\star}\in\Theta$, $(x(0),\tilde{x}(0))\in \Delta_{x}\mathcal{H}(0) \times \Delta_{\tilde{x}}\mathcal{H}(0)$ and for any $u\in \mathcal{M}_{\Delta_u}$ such that Assumption \ref{ass:PE_y_tilde} holds, the monitoring signal \eqref{eq:monitoring} is locally Lipschitz in $p$ on $\Theta$ and satisfies the following for all $t'\geq 0$ 
		\begin{equation}
				\underline{\alpha}(|\tilde{p}|) \leq \mu(p,t',t) \leq \epsilon_{\mu} + \bar{\alpha}(|\tilde{p}|), \qquad \forall t\geq t' + T, \label{eq:mu_bound}
		\end{equation}
	where $\tilde{p}:=p-p^{\star}$. \hfill $\Box$
\end{lem}
\begin{proof}
Property \eqref{eq:mu_bound} was proved in Lemma 2 in \cite{chong2015parameter}. We now proceed with proving that $\mu(p,t',t)$ is locally Lipschitz in $p$ on $\Theta$, for all $t\geq t'$. First, note that the monitoring signal in \eqref{eq:monitoring} can be written as 
\begin{equation}
\frac{d\mu(p,t',t)}{dt}= -\lambda \mu(p, t', t) + |\tilde{y}(p,t)|^2, \; \mu(p,t',t')=0. \label{eq:monitoring_2}
\end{equation} 
Also, since $h$ is continuously differentiable, $\tilde{y}(p,t)$ is continuously differentiable in $p$ if $\hat{x}(p,t)$ is. To this end, we augment the $\hat{x}(p,t)$-system, i.e. $\dot{\hat{x}(p,t)} = \hat{f}(\hat{x},p,u,y)$, with 
\begin{equation}
 \dot{p}=0, \qquad p(t')=p. \label{eq:p_system}
\end{equation} 
Since $\hat{f}$ is continuously differentiable, the right-hand side of the augmented system is continuously differentiable in $p$. Under these conditions, we conclude using the differentiability theorem in  \cite[Chapter 4.6]{arnold1992ordinary} that $\hat{x}(p,t)$ is a continuously differentiable function of its initial conditions $\hat{x}(p,t')$ and $p(0)=p$. 

Therefore, we have that $\tilde{y}(p,t)$ is continuously differentiable in $p$. By noting that the system \eqref{eq:monitoring_2} augmented with system \eqref{eq:p_system} has a continuously differentiable right-hand side, we conclude again using differentiability theorem in \cite[Chapter 4.6]{arnold1992ordinary} with similar arguments as before that $\mu(p,t',t)$ is locally Lipschitz in $p$ on $\Theta$.
\end{proof}

We are now ready to prove Theorem \ref{thm:direct}. Let $\Delta_x$, $\Delta_{\tilde{x}}$, $\Delta_u$, $\eta$ and $d^{\star}> 0$,
\begin{itemize}
		\item $k^{\star}:=3^{n_p -1}\left(\frac{3^{n_p(i+1)}-1}{3^{n_p}-1}\right)$ with $i$ satisfying $\frac{(n_p 3^{-2i})^{\frac{1}{2}}}{2} \leq  d^{\star}$ from Lemma \ref{lem:direct_rect}.
		\item Choose $\epsilon_{\mu}>0$ sufficiently small such that 
				\begin{equation}
						\underline{\alpha}^{-1}(2\epsilon_{\mu}) \leq \min\left\{ \eta, \frac{1}{2} \bar{\gamma}_{\tilde{x}}^{-1}(\eta) \right\}, \label{eq:thm_step}
				\end{equation}
				where $\underline{\alpha}$ and $\bar{\gamma}_{\tilde{x}}$ are a class $\mathcal{K}_{\infty}$ functions from Lemmas \ref{lem:mu} and \ref{lem:state_error}, respectively. Then, generate $T>0$ from Lemma \ref{lem:mu}. 
\end{itemize} 

Let $(x(0),\tilde{x}(0))\in \Delta_{x}\mathcal{H}(0) \times \Delta_{\tilde{x}}\mathcal{H}(0)$ and any $u\in\mathcal{M}_{\Delta_u}$. We consider any $t\geq 0$ as the systems \eqref{eq:system} and \eqref{eq:error_sys_general} are forward complete in view of Assumptions \ref{assum:bounded_x} and \ref{ass:obs_error_i}.  

Recall that the monitoring signal \eqref{eq:monitoring} satisfies the following by definition of the selection criterion in \eqref{eq:decision_monitor}, and consider the time interval $t\in[t_k, t_{k+1})$ with $k\geq k^{\star}$,
\begin{equation}
		\mu_{\sigma(t)}(t) = \underset{i}{\min} \; \mu_{i}(t) \leq \mu_{i^*}(t), \label{eq:mu_min_1}
\end{equation}
where $i^* \in \underset{i\in\{1,\dots,N_{\hat{\Theta}(k)}\}}{\arg \min} |p_i-p^{\star}|$. Combining with \eqref{eq:mu_min_1} with Lemma \ref{lem:mu}, 
\begin{align}
		\underline{\alpha}(|\tilde{p}_{\sigma(t)}(t)|) \leq & \mu_{\sigma(t)}(t)  \leq \mu_{i^*}(t) \leq \epsilon_{\mu}  + \bar{\alpha}(|\tilde{p}_{i^*}|), \nonumber \\ & \qquad \forall t\in[t_k+T, t_{k+1}),\, k\geq k^{\star}. \label{eq:mu_min_ineq}
\end{align}

By Lemma \ref{lem:direct_rect}, we have that $d(p^{\star},\hat{\Theta}(k)) \leq d^{\star}$, for all $k\geq k^{\star}$. Hence, we obtain for all $t \geq t_{k^*}+T$,
\begin{equation}
		|\tilde{p}_{\sigma(t)}(t)| \leq \underline{\alpha}^{-1} (\epsilon_{\mu}+ \bar{\alpha}(d^{\star})) \leq \underline{\alpha}^{-1}(2\epsilon_{\mu}) + \underline{\alpha}^{-1}(2\bar{\alpha}(d^{\star})), \label{eq:tilde_sigma}
\end{equation}
where the last inequality is obtained using the relation for any class $\mathcal{K}_{\infty}$ function $\alpha$, we have that $\alpha(a + b) \leq \alpha(2a) + \alpha(2b)$, for all $a$, $b\in\mathbb{R}_{\geq 0}$. Moreover, Lemma \ref{lem:state_error} gives the following bound on the chosen state estimation error
\begin{equation}
		|\tilde{x}_{\sigma(t)}(t)| \leq \bar{k} \exp(-\bar{\lambda}t) |\tilde{x}(0)| + \bar{\gamma}_{\tilde{x}}(|\tilde{p}_{\sigma(t)}|), \qquad \forall t\geq 0, \label{eq:tilde_x}
\end{equation}
where $\bar{\gamma}_{\tilde{x}}$ is a class $\mathcal{K}_{\infty}$ function. Therefore, we obtain \eqref{eq:conc} using \eqref{eq:tilde_sigma}, \eqref{eq:thm_step} and \eqref{eq:tilde_x}, by letting  $\nu_{\tilde{p}}(r):=\underline{\alpha}^{-1}(2\bar{\alpha}(r))$, $\gamma_{\tilde{x}}(r):=\bar{\gamma}_{\tilde{x}}(2 \nu_{\tilde{p}}(r))$ and $T^{\star}:=t_{k^*}+T$.  \hfill $\Box$


\section{Numerical simulations: a neural mass model} \label{sec:example}
We apply the setup on a neural mass model in \cite{jansen1995eeg}. This model describes the dynamics of the mean membrane potential (states) and the synaptic gain (parameters) of neuron populations, which is used to capture the elecroencephalogram (measurement) patterns related to various brain activities. By taking the states to be $x=(x_{11}, x_{12}, x_{21}, x_{22}, x_{31}, x_{32})\in\mathbb{R}^{6}$ and the \emph{unknown} parameter vector to be $p^{\star}=(p_{1}^{\star},p_{2}^{\star})\in\mathbb{R}^{2}$, which is known to reside in a compact set $\Theta:=[2,8]\times[22,28]$ which we normalise to a unit hypercube. Its dynamics is given by
\begin{align}
	\dot{x} & = A x + G(p^{\star}) \gamma(Hx) + B(p^{\star})\phi(u,y), \; y=Cx,	\label{eq:eg_dynamics}
\end{align}
where $A = \textrm{diag}(A_a, A_a, A_b)$, where $A_a = \begin{bmatrix} 0 & 1 \\ -a^2 & -2a \end{bmatrix}$ and $A_b = \begin{bmatrix} 0 & 1 \\ -b^2 & -2b \end{bmatrix}$, $G(p) = \begin{bmatrix} 0_{3\times 1} & 0_{3\times 1} \\ p_1 a c_2 & 0 \\ 0 & 0 \\ 0 & p_2bc_4 \end{bmatrix}$, $B(p)=\begin{bmatrix} 0 & 0 \\ p_1a & 0 \\ 0 & 0 \\ 0 & p_1 a \\ 0_{2\times 1} & 0_{2\times 1} \end{bmatrix}$, $H = \begin{bmatrix} c_1 & 0_{1\times 5} \\ c_3 & 0_{1\times 5} \end{bmatrix}$ and $C=\begin{bmatrix}0 & 0 & 1 & 0 & -1 & 0\end{bmatrix}$. We have used the notation $0_{r\times s}$ to denote an $r$ by $s$ matrix with all $0$ entries. The known parameters are $a=100$, $b=50$, $c_1=135$, $c_2=108$, $c_3=33.75$, $c_4=33.75$ are assumed to be known. The nonlinear terms are $\gamma = (S,S)$ and $\phi(u,y)=(S(y),u)$, where $S(v):=2e_0/(1+e^{r(v_0-v)})$, for $v\in\mathbb{R}$, with known constants $e_0=2.5$, $v_0=6$, $r=0.56$.

This model satisfies Assumption \ref{assum:bounded_x} because the matrix $A$ is Hurwitz and the nonlinear terms only contain a bounded input $u\in\mathcal{M}_{\Delta_u}$ and the function $S$ which is bounded. By also noting that the  $S$ is slope-restricted, we can employ Proposition 4 in \cite{chong2015parameter} to design our multi-observer \eqref{eq:multiobs} to satisfy Assumption \ref{ass:obs_error_i}. See \cite[Section VI-B]{chong2015parameter} for details.

With  $d^{\star}=0.8$, we calculate according to Lemma \ref{lem:direct_rect} that the termination iteration of the DIRECT algorithm is $k^*=6$. Other parameters for the algorithms are $\lambda = 0.05$ in \eqref{eq:monitoring}, sampling interval $T_d = 10$s and DIRECT search parameter $\epsilon=10^{-5}$. The results obtained are shown in Figures \ref{fig:res} and \ref{fig:res_error}. 

To compare DIRECT with the previous dynamic sampling scheme in \cite{chong2015parameter}, we ran both algorithms for $t_f=100$s with update instants at $t_k$, $k\in\{0,1,\dots,9\}$ and sampling interval $T_d=10$s. We compare the convergence time for a desired margin of the parameter estimation error, the average number of observers used during simulation time $[0,100]$s, the parameter estimation error at the end of the run-time $t_f$ as well as the corresponding normalised state estimation error.  We summarise this comparison in Table \ref{table:compare}, which shows that DIRECT achieves a significantly better parameter estimation accuracy. Namely, given the desired accuracy, DIRECT requires a fewer number of observers on average compared to the dynamic sampling policy in \cite{chong2015parameter}. The average number of observers used for DIRECT decreases as the run-time $t_f$ is increased since only one observer is used from $t=k^*T_d$ onwards, while the average number of observers used remains constant for the  dynamic sampling policy in \cite{chong2015parameter}. However, this trend does not necessarily translate to the state estimation error, due to the fact that the parameter mismatch gain function $\bar{\gamma}_{\tilde{x}}$ (c.f. Lemma \ref{lem:state_error}) can differ between observers.

\begin{figure*}[h!]
	\begin{center}
		\includegraphics[width=5.5cm]{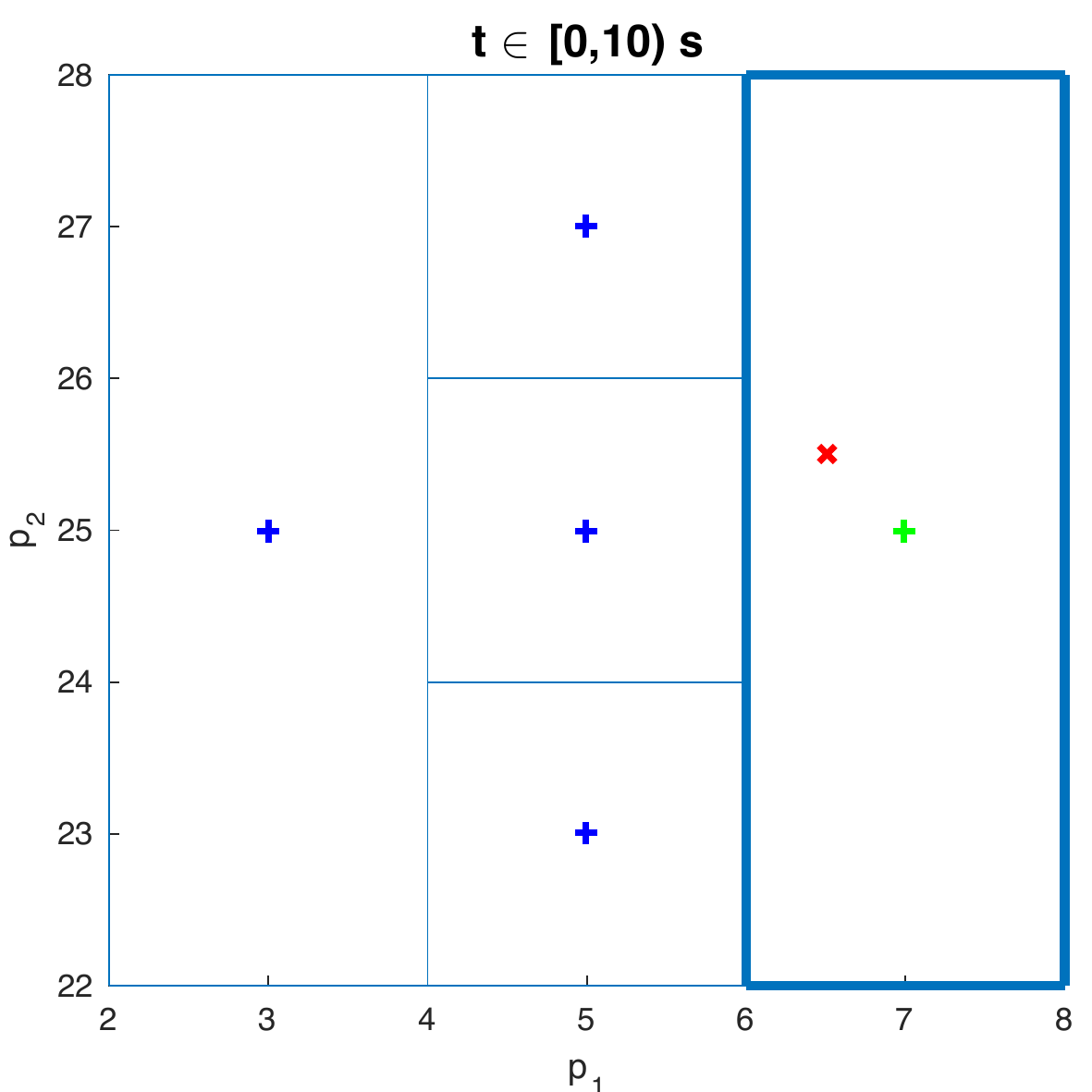}
		\includegraphics[width=5.5cm]{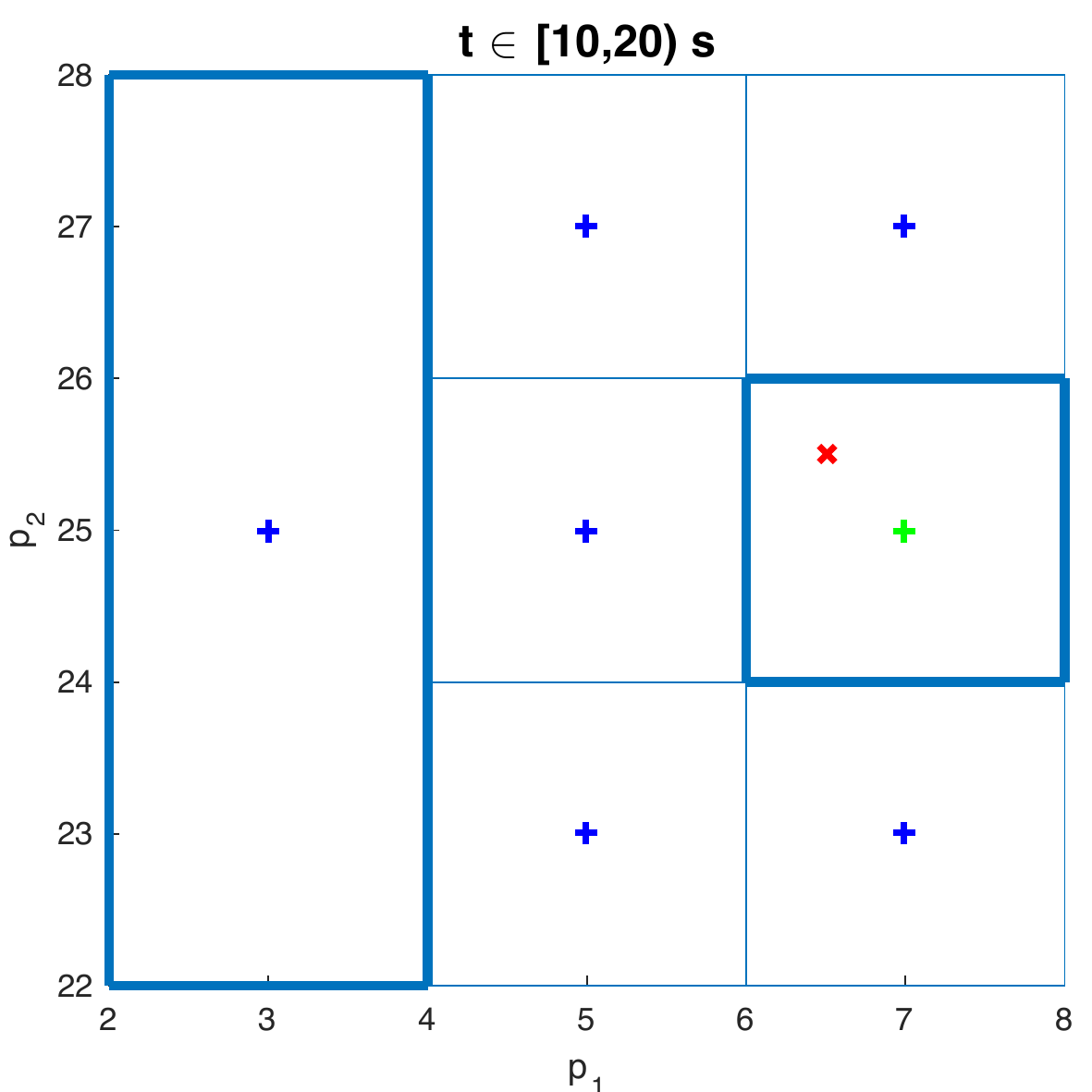}
		\includegraphics[width=5.5cm]{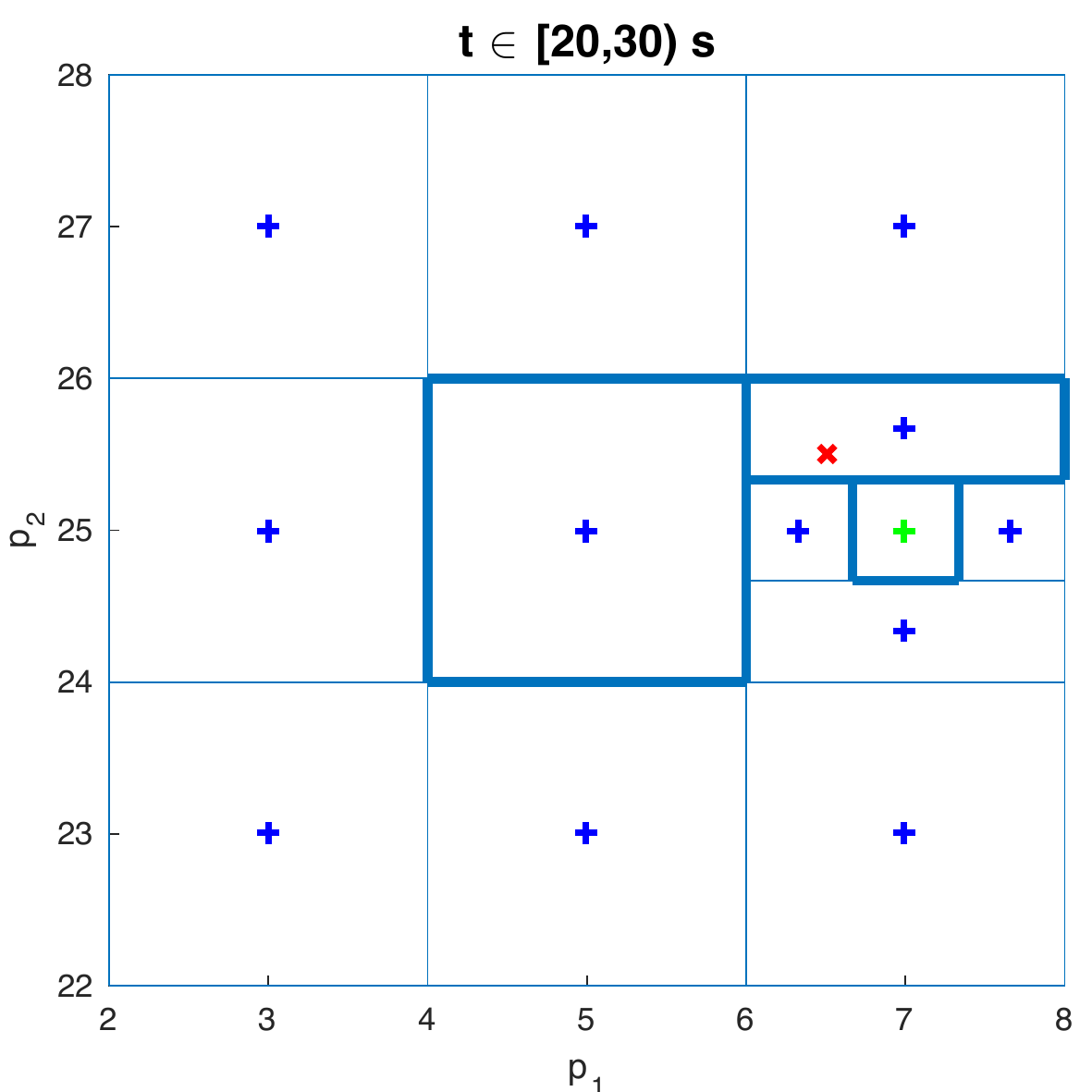} 
		\includegraphics[width=5.5cm]{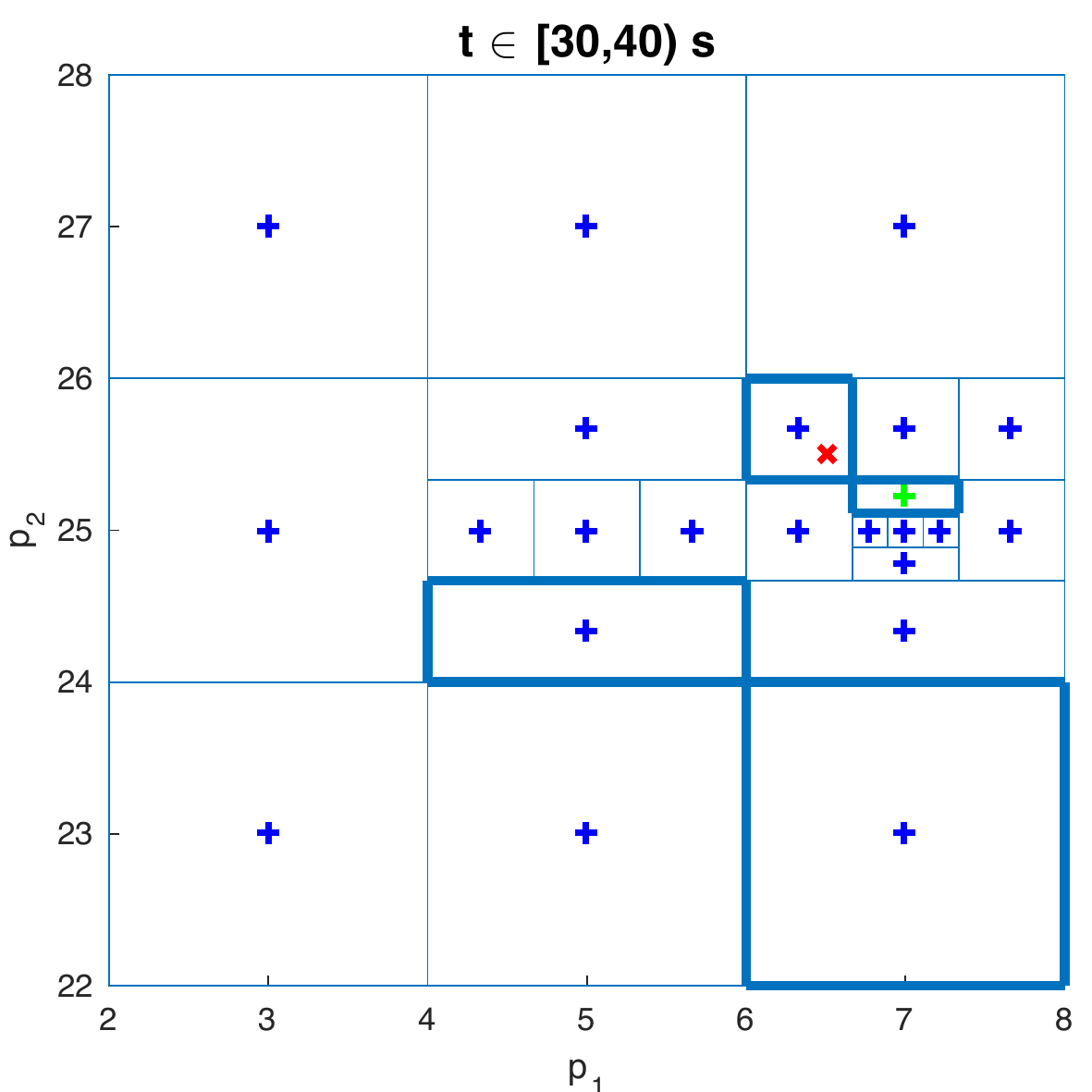}
		\includegraphics[width=5.5cm]{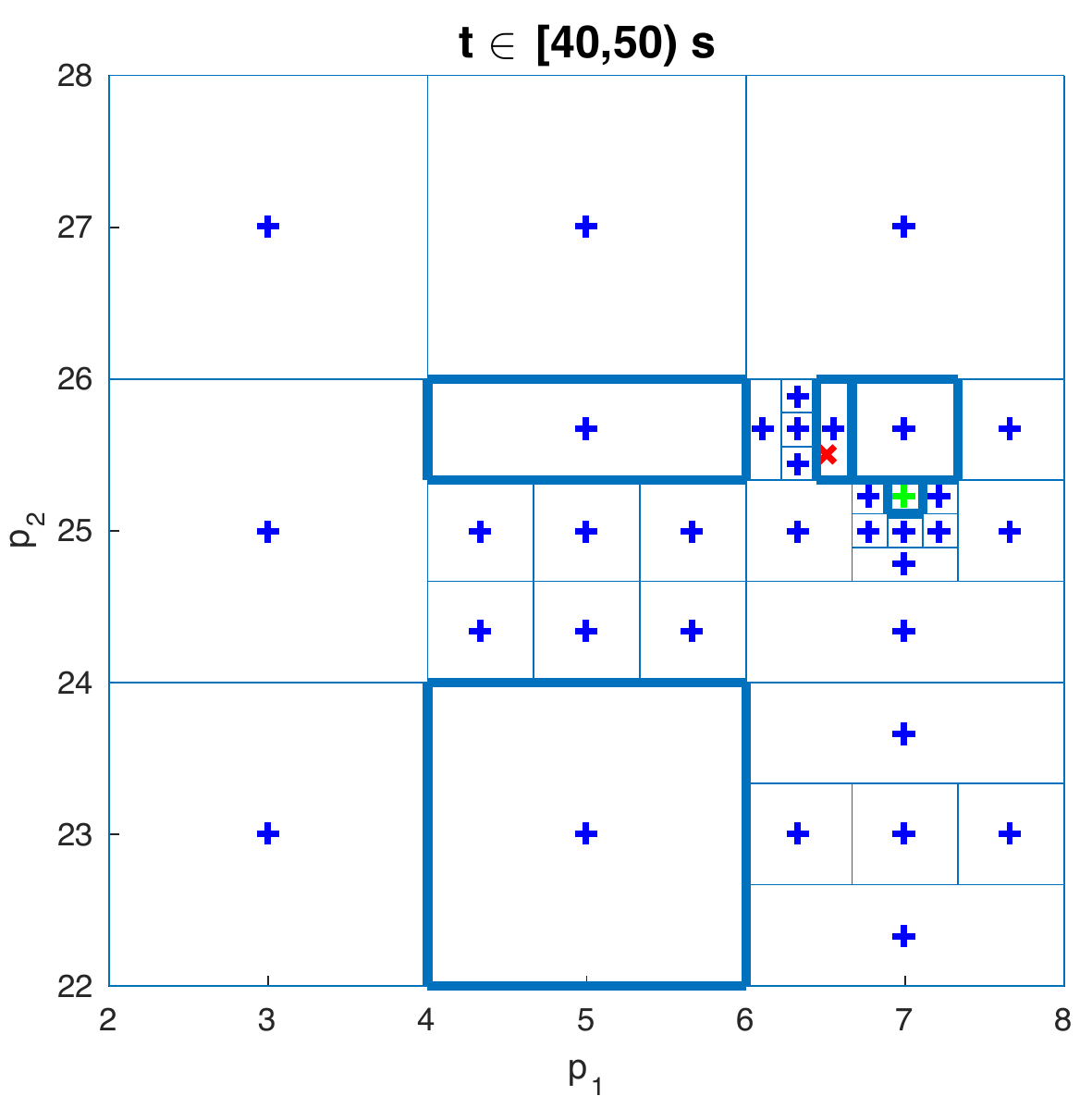}
		\includegraphics[width=5.5cm]{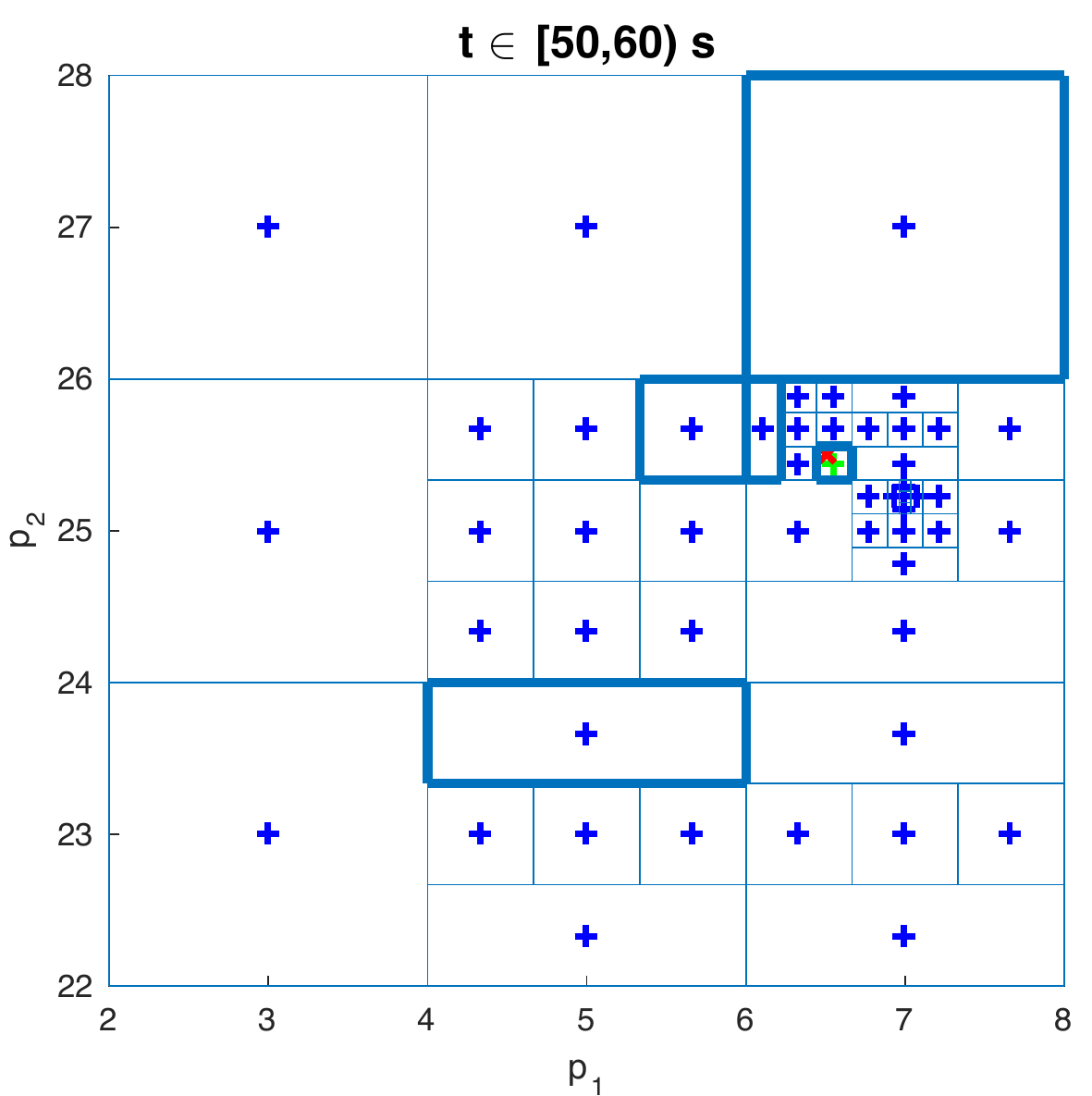}
	\end{center} \vspace{-1em}
	\caption{Legend: {\color{red}$\times$} is the true parameter $p^{\star}$; {\color{blue}$+$} are the sampled parameters;  {\color{green} $+$} is the final selected parameter in the interval $t\in[t_k,t_{k+1})$; the potentially optimal hyperectangles at each iteration is highlighted with thicker lines. \label{fig:res}  }
\end{figure*}

\begin{figure}[h!]
	\begin{center}
		\includegraphics[width=8.5cm]{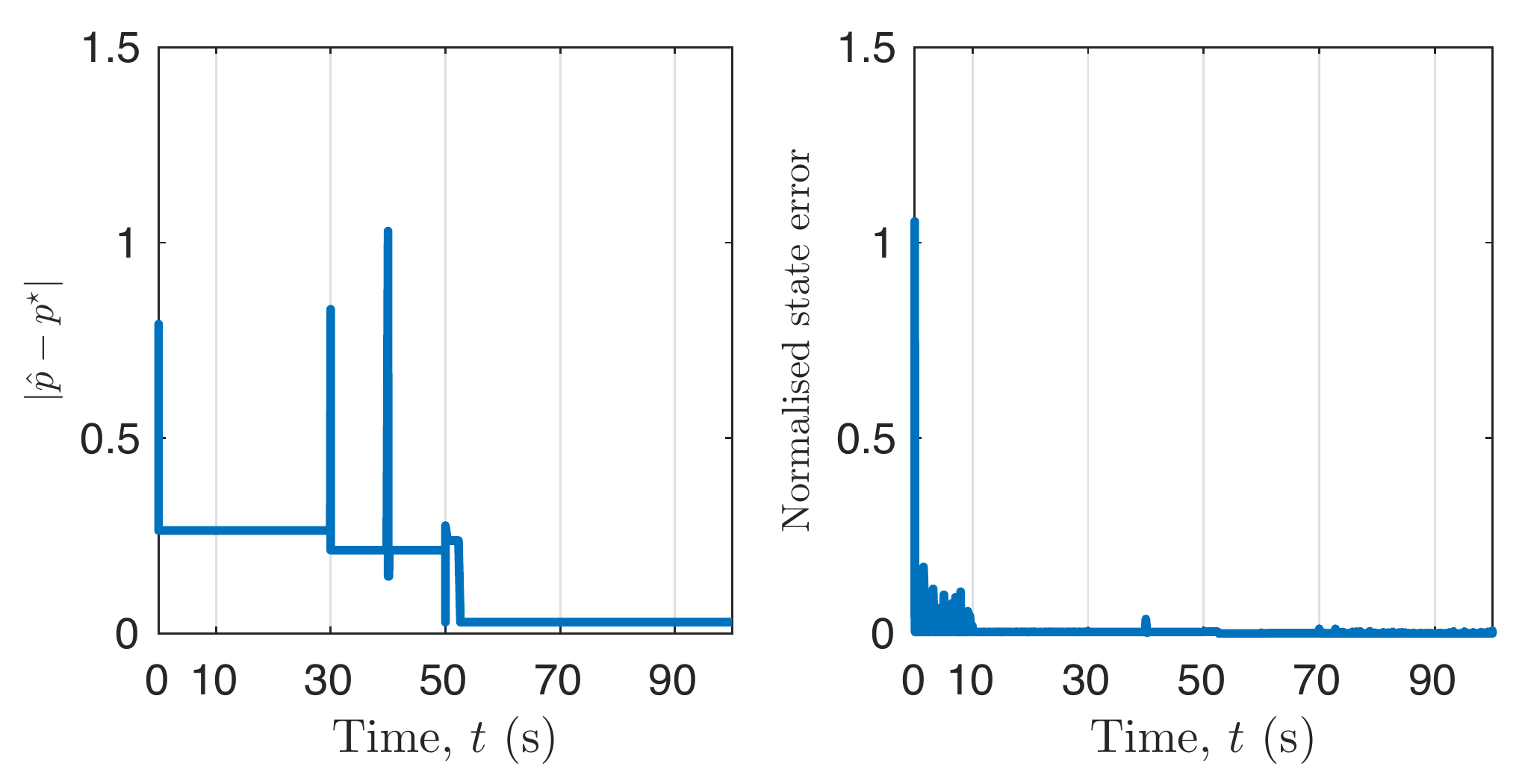}
	\end{center} \vspace{-2em}
	\caption{Parameter error $|\hat{p}(t)-p^{\star}|$ and normalised state estimation error $\frac{|\tilde{x}_{\sigma(t)}(t)|}{\max_{t} |x(t)| - \min |x(t)|}$. \label{fig:res_error}} \vspace{-1em}
\end{figure}

\begin{table}[h!] 
	\begin{center}
\begin{tabular}{p{3.8cm} p{1.5cm} p{1.9cm}}
	
	\toprule
	 & \textbf{DIRECT}  & \textbf{Dynamic policy in \cite{chong2015parameter}} \\
	\midrule
	\textbf{Convergence time $T^\star$ such that $|\tilde{p}_{\sigma(t)}(t)|\leq 0.72$, $\forall t \geq T^\star$.} & $45$s & $90$s \\
	\hline
	\textbf{Average number of observers  $$\frac{\sum_{k=0}^{9} N(t_{k}) }{\lceil t_{f}/T_d\rceil} $$} \newline where $N(t_k)$ is the number of observers used at $t\in[t_{k},t_{k+1})$, $k\in\{0,\dots,9\}$.   & $13.8$ & $16$ \\
	\hline
	\textbf{Parameter estimation error $|\tilde{p}_{\sigma(t_f)}(t_f)|$} & $0.03$ & $1.04$ \\
	\hline
	\textbf{Normalised state estimation error \newline  $$\frac{|\tilde{x}_{\sigma(t_{f})}(t_{f})|}{\underset{t\in[0,t_f]}{\max}|x(t)|-\underset{t\in[0,t_f]}{\min}|x(t)|}$$} & $0.0043$  & $0.0025$ \\
	\bottomrule
\end{tabular}
	\end{center}
	\caption{Numerical results comparing the supervisory observer with DIRECT vs. the previous dynamic policy in \cite{chong2015parameter} with $t_f=100$s.    \label{table:compare} }	
\end{table}

\section{Conclusions and future work} \label{sec:conc}
We have used the DIRECT optimisation algorithm to generate the samples needed to implement a supervisory observer, as proposed in \cite{chong2015parameter}. By doing so, we are able to overcome one of the issues of \cite{chong2015parameter}, which is the selection of the number of samples. Indeed, DIRECT automatically generates samples in the parameter set to improve estimation and it stops iterating after a given time, which is easy to compute. Afterwards, a single observer is implemented, which helps ease the computational complexity of \cite{chong2015parameter}. Results have been illustrated on a numerical example of a neural mass model. Future work includes providing robustness guarantees with respect to measurement noise and unmodelled dynamics.


\bibliographystyle{plain}
\bibliography{supervisory-direct.bib}

\end{document}